\documentclass[11pt,a4paper]{article}

\usepackage{amsmath}
\usepackage{amsfonts}
\usepackage{mathrsfs}
\usepackage{subfigure}
\usepackage{epsfig}
\usepackage{pstricks}
\usepackage{psfrag}

\newtheorem{prop}{Proposition}

\newtheorem{corollary}{Corollary}
\newtheorem{example}{Example}
\newenvironment{proof}{
    {\bf Proof}}{\hbox{\ }\hfill$|||$
}

\newcommand{\hlf}{\frac{1}{2}}

\newcommand{\figref}[1]{Fig.~\ref{#1}}
\newcommand{\exref}[1]{Example~\ref{#1}}
\newcommand{\corref}[1]{Corollary~\ref{#1}}
\newcommand{\secref}[1]{Section~\ref{#1}}

\newcommand{\propref}[1]{Proposition~\ref{#1}}

\newcommand{\smt}[2]{\Bigl[
 \begin{smallmatrix} #1\\ #2
 \end{smallmatrix}\Bigr]}

\newcommand{\gc}{\gamma}

\newcommand{\eps}{\epsilon}

\newcommand{\Chi}{\chi}

\newcommand{\bA}{\mathbf{A}}   
\newcommand{\bD}{\mathbf{D}}   
\newcommand{\bI}{\mathbf{I}}   
\newcommand{\bJ}{\mathbf{J}}   
\newcommand{\be}{\mathbf{e}}   
   
\newcommand{\bj}{\mathbf{j}}   
   
\newcommand{\br}{\mathbf{r}}

\newcommand{\R}{\mathbb{R}}
\newcommand{\Z}{\mathbb{Z}}

\newcommand{\AL}{{\mathcal{A}}}
\newcommand{\BL}{{\mathcal{B}}}
\newcommand{\DL}{{\mathcal{D}}}
\newcommand{\EL}{{\mathcal{E}}}

\newcommand{\ifn}{indicator function}

\newcommand{\sT}{T}    
\newcommand{\Tbase}{\sT}   
\newcommand{\sH}{{ \cal{H}}}
\newcommand{\Dim}{k}
\newcommand{\cl}{c}  
\newcommand{\clp}{\cl'}  

\begin{document}%
\title{Refinability of splines from lattice Voronoi cells}

\author{J\"org Peters}

\maketitle

\begin{abstract}
Splines can be constructed by convolving the \ifn\ of the Voronoi cell
of a lattice.
This paper presents simple criteria that imply that only a small subset
of such spline families can be refined:
essentially the well-known box splines and tensor-product splines.
Among the many non-refinable constructions are hex-splines and
their generalization to non-Cartesian lattices.
An example shows how non-refinable splines can exhibit increased 
approximation error upon refinement of the lattice.
\end{abstract}

\section{Introduction}
Univariate B-splines are defined by repeated convolution, 
starting with the \ifn s of a partition of the real line
by knots (An \ifn\ takes on the value one on the interval but
is zero otherwise).
This construction implies local support and 
a number of desirable properties (see \cite{Boor:1978:PGS,deboor87e})
that have made B-splines the representation of choice 
in modeling and analysis. In particular, B-splines can 
be exactly represented as a linear combination of B-splines with
a finer knot sequence. This refinability is a key ingredient of
multi-resolution, adaptive and sparse representation of data.
 
By tensoring univariate B-splines, we can obtain 
on Cartesian grids in any dimension.
For uniform knots, box-splines \cite{deboor93box} generalize this construction
by allowing convolution directions other than the orthogonal ones of tensoring. 
This is not to say that box-spline convolution directions 
are arbitrary; to be practically useful,
the directions need to be compatible with the lattice on which the 
spline is shifted, so that only a small number but sufficiently many
lattice-shifts of the box-spline overlap at every point.

As a most prominent example in two variables, 
the three direction box-spline forms a
function with hexagonal footprint. The function is 
called `hat function' and consists of six linear function pieces over
the constituent triangles.
Shifts over an equilateral triangulation add up to 1. 
Convolution of this hat function with itself results in a twice
continuously differentiable function of degree 4 
and $m$-fold convolution yields a function of degree $3m-2$
with smoothness $C^{2m}$.
Since this progression skips odd orders of smoothness,
van der Ville et al. \cite{van:04} proposed to directly
convolve the \ifn\ of the hexagon and
build splines on the corresponding tessellation of the plane.
They went on to show that the resulting hex-splines share a number 
of desirable properties familiar from box-splines. But the authors
did not settle whether the splines were \emph{refinable}
\cite{perscomBanff},
i.e.\ whether hex-splines of the given hexagonal tessellation
$\Tbase$ can be represented as linear combinations of hex-splines based
on a finer-scale hexagonal tessellation, say $\hlf\Tbase$.
Refinability is important in practice since 
it guarantees monotonically decreasing approximation error
as the scale of the tessellation refined. 
Moreover refinability is needed to locally adapt the space
to features of higher frequency, a pre-requisite for 
multi-resolution analysis.

\begin{itemize}
\item
This paper presents simple criteria on a lattice that need to hold
in order for shift-invariant functions on that lattice 
to be represented as linear combinations of 
piecewise constant shift-invariant functions on 
the smaller-scaled copy of the lattice.
\end{itemize}
For example, the lattice must contain, for every of cell facet, 
the whole plane containing it.
Therefore, requiring refinability, even of just the constant spline,
strongly restricts allowable lattices. 
\begin{itemize}
\item
In contrast to tensor-product and box splines,
we show that hex-splines and similar constructions 
can only be scaled, but not refined: scaled hex-spline spaces are not nested.
\item
A concrete example illustrates that non-refinable spaces 
can exhibit increased 
approximation error on a finer-scaled copy of the underlying tessellation
(see \exref{ex:error}).
\item
The analysis is extends to overcomplete spaces.
\end{itemize}

{\bf Overview.}
\secref{sec:lit} reviews lattices,
hex-splines and their generalizations.
\secref{sec:nonref} exhibits two simple criteria 
for testing whether a lattice can support a refinable
space of splines via convolution of \ifn s.
\secref{sec:overcomplete} extends this investigation to
overcomplete coverings by combining several families of 
\ifn s shifted by less than the lattice spacing. 
\secref{sec:discuss} illustrates the importance of refinability.

\section{Splines from lattice Voronoi cells}
\label{sec:lit}
A $\Dim$-dimensional lattice is a discrete subgroup of full rank in
a $\Dim$-dimensional Euclidean vector space.
A lattice may alternatively be viewed as a tessellation of space
by identical cells. The Euclidean plane admits three highly symmetric
tessellations into
equilateral triangles, squares, or hexagons respectively.
Convolution starting with the \ifn\ on any of these 
polygons yields a hierarchy of spline functions of local support.
The regular partition into squares 
gives rise to uniform tensor-product B-splines and
the regular triangulation and its hexagonal dual to box splines. 
An interesting and natural complement, 
to convolve the \ifn\ $H$ of the hexagon with itself,
was developed and analyzed by van De Ville et al. \cite{van:04}.
This yields a family of splines, of smoothness $n-1$ supported on 
a local $n+1$-neighborhood, that the authors named hex-splines.
van De Ville et al. observed that hexagons are Voronoi cells 
of a lattice and that the cell can
be split into three quadrilaterals using one of two choices of the 
central split. Thus $H$ can be viewed as the union 
of three constant box splines \cite{deboor93box},
an approach that was worked out more generally for the FCC lattice
by \cite{Kimperscom} and in more generality in 
\cite{journals/tsp/MirzargarE10}.
\cite{van:04} compares tensor-product splines and hex-splines
and presents a Fourier transform. The transform
allows determining an $L^2$ approximation order 
with emphasis on low frequencies, as a combination of 
projection into the hex-spline space and a quasi-interpolation error.
\cite{condat:05} derived quasi-interpolation formulas 
and showed promising results when applying
hex-splines to reconstruction in image processing.

\section{Refinability constraints for lattice Voronoi cells }
\label{sec:nonref} 
Given a tessellation $\Tbase$ of $\R^k$, we denote by $\Chi(\Tbase)$ 
the space of \ifn s over the cells of $\Tbase$
and by $\Chi(T^{1})$ the space of \ifn s
on the smaller-scale copy $T^{1}$ of $\Tbase$.
The space $\Chi(\Tbase)$ is refinable if
each \ifn\ in $\Chi(\Tbase)$
can be represented as linear combinations of functions in $\Chi(T^{1})$.

Testing whether a tessellation admits a refinable space of functions
requires off hand solving for weights such that a linear combination
of elements in $\Chi(T^{1})$ with these weights reproduces
each element in $\Chi(\Tbase)$. 
\propref{prop:straddle} below provides
a much simpler necessary condition that avoids such algebraic analysis.
While we are interested in shift-invariant tessellations,
\propref{prop:straddle} applies more generally. 

\begin{prop}
Let $\Tbase$ be a tessellation of $\R^k$
and $T^{1}$ its smaller-scale copy. Then  $\Chi(\Tbase)$ is
refinable only if every tessellation facet of $\Tbase$ is representable 
as a union of tessellation facets of $T^{1}$.
\label{prop:straddle}
\end{prop}

\begin{proof}
Since $T^1$ is a tessellation, its cells do not overlap.
Therefore, if a facet of a cell $\cl$ in $\Tbase$ is not 
a union of tessellation facets of $T^{1}$ then a cell $\clp$ of $T^{1}$
must cross this facet. Let $H^1 \in \Chi(T^{1})$ be the \ifn\ of $\clp$
and $H$ the \ifn\ of $\cl$.
Then, in order to reproduce $H$,
$H^1$ would have to take on both values $0$ and $1$.
\end{proof}

Scaling of the tessellation
transforms this criterion to a less local one.

\begin{prop}
For $\Chi(T)$ to be refinable, a tessellation $T$ of $\R^k$
must contain, for each facet, 
a whole plane of the same dimension parallel to it.
\label{prop:ext}
\end{prop}
\begin{proof}
Considering ever coarser tessellations, \propref{prop:straddle}
implies that arbitrarily large extensions of each type of facet
must be a union of tessellation facets of $\Tbase$.
\end{proof}

The lattice structure implies that each such plane is replicated 
at all lattice points.

\begin{corollary}
For $\Chi(T)$ to be refinable,
if $T$ is a lattice, $T$ must contain, for every facet, 
the whole plane of the same dimension that contains it.
\label{cor:ext}
\end{corollary}

By inspection of the three regular tessellations of the plane,
only the Cartesian grid and the uniform triangulation satisfy
\corref{cor:ext}, but not the partition into hexagons. 
\begin{corollary}
Hex splines are not refinable.
\end{corollary}
We can generalize this observation by simplifying the inspection 
criterion.

\def\swid{0.35\linewidth}
\begin{figure}[h]
   \centering
   \psfrag{f1}{$f_1$}
   \psfrag{f2}{$f_2$}
   \psfrag{A}{$\cl$}
   \psfrag{B}{$\clp$}
   \psfrag{sym}{$F_2$}
   \psfrag{int}{$F_1$}
   \includegraphics[width=\swid]{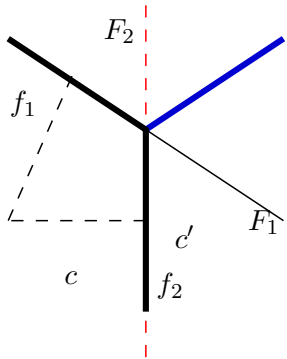}
   \caption{ A pair of abutting facets, whose normals (dashed) have a strictly 
     positive inner product, does not allow for refinability.
   }
   \label{fig:obtuse}
\end{figure}
\begin{prop}
\label{prop:obtuse}
Consider a lattice $\Tbase$ such that the reflection $\clp$
of a cell $\cl$ across one of $\cl$'s facets is again a cell of $\Tbase$.
If abutting cell facets of $\cl$ meet with an obtuse angle
then $\Chi(\Tbase)$ is not refinable.
\end{prop}
\begin{proof}
Let $\cl$ be a cell of $\Tbase$ with facets $f_1$ and $f_2$ that 
join with an obtuse angle (see \figref{fig:obtuse}).
Let $\clp$ be the cell sharing $f_2$. By the reflection symmetry,
$\clp$ also has an obtuse angle between $f_2$ and the mirror image of 
$f_1$ across the plane $F_2$ through $f_2$.
But on the side of $\clp$, $F_1$ forms an accute angle with $f_2$
and therefore intersects the interior of cell $\clp$.
Since cells can not be split, $F_1$ cannot be part of $\Tbase$ and
the claim then follows from \corref{cor:ext}.
\end{proof}

\propref{prop:obtuse} allows us to quickly decide 
which of the (regular crystallographic) root lattices 
$\AL_n$, $\AL^*_n$, 
$\BL_n$, 
$\DL_n$, $\DL^*_n$, 
$\EL_j$, $j=6,7,8$ \cite{conway98}
are suitable for building splines by convolution of
their Voronoi cells.
\begin{corollary}
Splines obtained by convolution of Voronoi cells of 
regular crystallographic root lattices are not refinable,
except for the Cartesian grid and the 
bivariate lattice with triangular Voronoi cells.
\label{cor:voronoi}
\end{corollary}
\begin{proof}
We test whether the Voronoi cells of the root lattices contain 
a pair of faces that meet with an obtuse angle. 
We may assume that one cell center is at the origin and take the 
inner product of the position vectors of two adjacent nearest neighbors,
as identified by their root system;
if the product is strictly positive, the corresponding Voronoi faces 
meet with an obtuse angle.

The $\AL_n$ lattice is traditionally defined via an embedding 
in $\R^{n+1}$, $n>1$. Alternatively,
Theorem 1 of \cite{kim:2010:andual} gives a 
convenient geometric construction in $\R^n$ via the 
$n\times n$ generator matrix $\bA_n := \bI_n + \frac{c_n}{n}\bJ_n$,
where $\bI_n$ is the identity matrix, $\bJ_n$ the $n\times n$
matrix of ones and $c_n := \frac{-1+\sqrt{n+1}}{n}$. 
Denoting the $i$th coordinate vector by $\be_i$,
we choose $\be_1$ and $\be_1+\be_2$ on 
the Cartesian grid, and map them via $\bA_n$ to the 
nearest $\AL_n$ neighbors of the origin. 
The inner product of the images of $\be_1$ and $\be_1+\be_2$ is
\begin{equation*}
   \bA_n\be_1 \cdot \bA_n(\be_1+\be_2) 
   =
   \frac{n+4c_n+c_n^2}{n}
   =
   \frac{2}{n}(n+\sqrt{n+1} - 1) > 0.
\end{equation*}
For the $\AL^*_n$ lattice, the computation is identical except that 
$c_n := \frac{-1+\frac{1}{\sqrt{n+1}}}{n}$.
The inner product is 
$\frac{1}{n(n+1)}(n^2 -2n-2+2\sqrt{n+1})) > 0$.

For the $\DL_n$ lattice, defined in $n\ge 3$ dimensions, 
the generator matrix is
$
   \bD_n :=
   \begin{bmatrix}
   \bI_{n-1} & - \be^{n-1}_{n-1} \\
   -\bj^t_{n-1} & - 1 \\
   \end{bmatrix}
$
(see e.g.\  Section 7 of \cite{Kim:2011:SBS})
and
\begin{equation*}
   \bD_n\be_1 \cdot \bD_n(\be_1+\be_2) 
   = 3 > 0.
\end{equation*}
Since $\bD^{-t}_n$ is the generator of $\DL^*_n$, the inner 
product for $\DL^*_n$ is  $2$.

For $\BL_n$, the Cartesian cube lattice has an inner product $0$
identifying its spline constructions as potentially refinable
(which indeed they are). On the other hand,
splitting each the cube by adding a diagonal direction
\cite{Kim:SBL:2008} yields the inner product 
$\be_1 \cdot \bj = 1$.

For $\EL_6$, we select the root vectors 
$(1,1,0,0,0,0) $ and $(1,1,1,1,1,\sqrt{3})/2$ with inner product $1$.
For $\EL_7$, we select the root vectors 
$(1,1,0,0,0,0,0) $ and $(1,1,1,1,1,1,\sqrt{2})/2$ with inner product $1$.
For $\EL_8$, we select the root vectors 
$(1,1,0,0,0,0,0,0) $ and $\bj_8/2$ with inner product $1$.

Equilateral triangular Voronoi cells in $\R^2$ yield the inner product 
$-\frac{-1}{2}$.

\end{proof}

\section{Overcomplete hex-spline spaces}
\label{sec:overcomplete}
Since the evaluation of hex-splines by convolving 
three families of box splines already leads to a large number of terms, 
it is reasonable to investigate whether redundant superposition
make hex-splines refinable.
To test whether we can build refinable frames, 
let $\{\Tbase_j\}_{j=1..J}$ be a family of tessellations obtained by
shifting $\Tbase$ by less than the lattice spacing
so that their union covers $\R^k$ $J$-fold.
The next example makes this concrete for $J=3$.

\begin{example}
\label{ex:lozenge}
\begin{figure}[!ht]
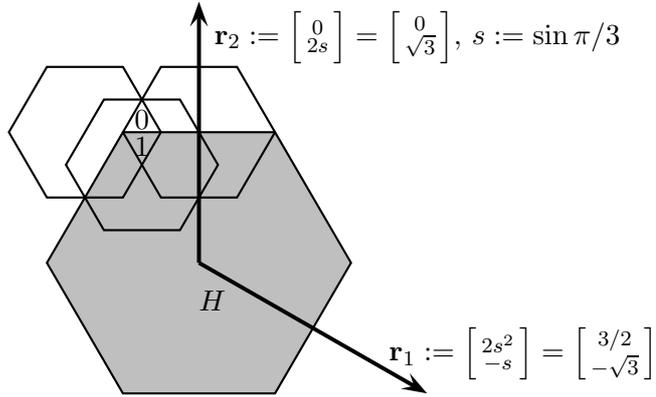

\centering
\psset{unit=1.0cm} \pspicture(-3,-2)(3,3)
\pspolygon[fillstyle=solid,fillcolor=lightgray](2,0)(1,-1.732)(-1,-1.732)(-2,0)(-1,1.732)(1,1.732)
\def\ss{0.8660254040}
\def\s4{0.8660254040}
\psline[linewidth=1.5pt,arrowscale=1.5]{->}(0,0)(0,3.464) 
\psline[linewidth=1.5pt,arrowscale=1.5]{->}(0,0)(3,-1.732) 
\rput[l](0.2,3){$\br_2 := \smt{0}{2s} = \smt{0}{\sqrt{3}}$,\ $s := \sin \pi/3$}
\rput[l](2.5,-1.2){$\br_1 := \smt{2s^2}{-s} = \smt{3/2}{-\sqrt{3}}$}
\rput[l](0,-0.5){$H$}
\pspolygon[fillstyle=none](0.5,2.598)(-0.5,2.598)(-1,1.732)(-0.5,0.866)(0.5,0.866)(1,1.732)
\pspolygon[fillstyle=none](-1.0,2.598)(-2,2.598)(-2.5,1.732)(-2,0.866)(-1,0.866)(-.5,1.732)
\pspolygon[fillstyle=none](-.25,2.165)(-1.25,2.165)(-1.75,1.3)(-1.25,0.433)(-.25,0.433)(0.25,1.299)
\rput(-0.75,1.55){1} \rput(-0.75,1.9){0}
\endpspicture
\caption{ A lozenge-shaped pair of triangles (with markers $0$ and $1$)
is in the common support of 
three half-scaled translated copies
of the grey hexagon and not in the support of other half-scaled hexagons.
Since the pair straddles the boundary of the hexagon,
any linear combination of the three \ifn s needs to be both 1 and 0. 
}
\label{fig:hexref}
\end{figure}
Denote by $\Tbase_3$ a tessellation of the plane into unit-sized hexagons
and by $\Tbase_1$ and $\Tbase_2$ its shifts by $\hlf\br_1$ and $\hlf\br_2$
(see \figref{fig:hexref}).
Let $H(x)$ be the \ifn\ of the unit hexagon of $\Tbase_3$
centered at the origin.
Consider the three \ifn s of 
hexagons of the $\hlf$-scaled tessellations shown in \figref{fig:hexref}.
Since the three hexagons supporting the three functions intersect in a 
pair of triangles, any linear combination of the functions has the same value 
on both triangles. But since the pair straddles the 
boundary of $H(x)$ the value on one must be $0$ and
the value on the other $1$ implying that the joint space is not
refinable.
\end{example}

The example points to a simple extension of \corref{cor:ext}.

\begin{corollary}
Consider a family $\{\Tbase_j\}_{j=1..J}$ of tessellations
each covering $\R^k$.
Then the space $\bigcup_{j=1..J}\Chi(\Tbase_j)$ is refinable only
if the tessellation obtained by intersecting $\{\Tbase_j\}$ 
does not contain a cell straddling a cell facet of any
coarser-scaled copy of any $\Tbase_j$.
\label{cor:family}
\end{corollary}

The main argument concerning straddling cells
applies to more general tessellations than 
shifts of a single tessellation $\Tbase$.

Extending the train of thought, the following \exref{ex:propagate}
shows that families without straddling triangle pairs
need not yield a refinable space of \ifn s either.
 
\begin{figure}[h]
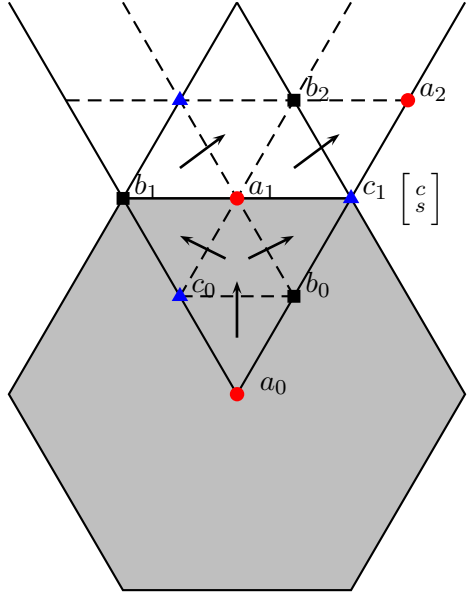

\centering
\psset{unit=1.5cm} \pspicture(-2,-2)(2,4)
\pspolygon[fillstyle=solid,fillcolor=lightgray](2,0)(1,-1.732)(-1,-1.732)(-2,0)(-1,1.732)(1,1.732)
\psline(-1,1.732)(1,1.732) 
\psline(0,3.464)(-1, 1.732)
\psline(2,3.464)(0,0)
\psline(-2,3.464)(0,0) 
\psline(0,3.464)(1,1.732)
\psline[linecolor=black,linestyle=dashed](-1.5,2.598)(1.5,2.598)
\psline[linecolor=black,linestyle=dashed](-0.5,0.866)(0.5,0.866)
\psline[linecolor=black,linestyle=dashed](1.0,3.464)(-0.5,0.866)
\psline[linecolor=black,linestyle=dashed](-1.0,3.464)(0.5,0.866)
\psdots[linecolor=red,dotscale=1.5](0,0)(0,1.732)(1.5,2.598) 
\psdots[linecolor=blue,dotstyle=triangle*,dotscale=1.5](-0.5,0.866)(1,1.732)(-0.5,2.598)
\psdots[linecolor=black,dotstyle=square*,dotscale=1.5](0.5,0.866)(-1,1.732)(0.5,2.598)
\psline[linewidth=1pt,arrowscale=1]{->}(0,0.5)(0,1.0)
\psline[linewidth=1pt,arrowscale=1]{->}(0.1,1.2)(0.5,1.4)
\psline[linewidth=1pt,arrowscale=1]{->}(-0.1,1.2)(-0.5,1.4)
\psline[linewidth=1pt,arrowscale=1]{->}(-0.5,2.0)(-0.1,2.3)
\psline[linewidth=1pt,arrowscale=1]{->}( 0.5,2.0)( 0.9,2.3)
\rput[bl](0.2,0){$a_0$}
\rput[bl](0.6,0.866){$b_0$}
\rput[bl](-0.4,0.866){$c_0$}
\rput[bl](1.1,1.732){$c_1$}
\rput[bl](1.4,1.5){$\smt{c}{s}$}
\rput[bl](0.1,1.732){$a_1$} 
\rput[bl](-0.9,1.732){$b_1$}
\rput[bl](1.6,2.598){$a_2$} \rput[bl](0.6,2.598){$b_2$} 
\endpspicture
\caption{Propagation of values via neighboring triangles
   that share an edge. This yields the contradiction that
   $a_0+b_0+c_0=1$ (center) and $c_1+b_2+a_2=0$ (upper right)
   since the propagation
   implies $a_i=a_0$, $b_i=b_0$ and $c_i=c_0$ for all $i$.
}
\label{fig:propagate}
\end{figure}
\begin{example}
\label{ex:propagate}
Consider shifts 
\begin{equation*}
   H_2(x) :=  H(x-\smt{c}{s}), \quad
   H_3(x) :=  H(x-\smt{-c}{s}), \quad
   c := \cos \frac{\pi}{3},\
   s := \sin \frac{\pi}{3}
\end{equation*}
of the \ifn\  $H(x)$.
The three corresponding tessellations now intersect only in single triangles
so that the scenario of \corref{cor:family} does not apply.
However, an algebraic argument with a simple geometric interpretation
proves lack of refinability.

We want to find scalar, real-valued 
coefficients $a_i$, $b_i$ and $c_i$, $i=(i_1,i_2) \in \Z^2$
such that the following refinement equation holds:
\begin{align*}
    H(x) &= 
    \sum_{i=(i_!,i_2) \in \Z^2}
    a_i H(2x-\gc_i)  + b_i H_2(2x-\gc_i) + c_i H_3(2x-\gc_i), 
    \\
    \gc_i &:= \frac{i_1}{2}\smt{c}{s}+\frac{i_2}{2}\smt{-c}{s}.
\end{align*}
We associate the coefficients with the center of its support hexagon.
Observe then that, when two triangles share an edge and $H(x)$
has the same value $v \in {0,1}$ on both triangles,
then the coefficients at the two non-shared vertices must be equal.
For example $a_0+b_0+c_0=v=a_1+b_0+c_0$ implies $a_0=a_1$.
As indicated by the arrows in \figref{fig:propagate}, the coefficients are 
therefore propagated,
separately inside and outside the support of $H$.
This contradicts the refinement equation in that both
$a_0+b_0+c_0=1$ and $a_2+b_2+c_1= a_0+b_0+c_0=0$ must hold.
\end{example}
So even the two natural extensions to overcomplete spaces of shifted
hex-splines do not afford refinability.

The propagation argument generalizes to face-sharing simplices
in any dimension.
And it generalizes from binary to $m$-ary refinement.

\section{Importance of Refinability}
\label{sec:discuss}
Why do we care about refinability and nestedness of spaces?
Approximation order is well-defined even 
for sequences of spaces that are not nested.
For example, the elegant Fourier-based 
estimates of \cite{van:04} show that hex-splines resulting from
$m$ convolutions have, for low frequencies,
an $L^2$ approximation order of $m$.
But approximation order is concerned with asymptotic estimates.
In practice one is more interested in predicting approximation error.

The following example shows why, for predicting the 
approximation error, nested spaces are highly desirable.

\begin{example}
\label{ex:error}
Let $\sH^i$ be the space of \ifn s over a
regular tessellation by hexagons of diameter $2^{-i}$ 
and such that, at each level of refinement, the origin
is the center of one hexagon.
Denote by $H$ the \ifn\ in $\sH^0$  whose hexagon is
centered at the origin.
Let $f$ be a $C^1$ function obtained by smoothing out $H$,
say by a degree 3 Hermite interpolant,
over a distance of at most $2\eps$ from the boundary of the 
hexagon. 
\def\swid{0.25\linewidth}
\begin{figure}[!ht]
\centering
   \subfigure[natural and man-made hex-tilings]{
   \includegraphics[width=\swid]{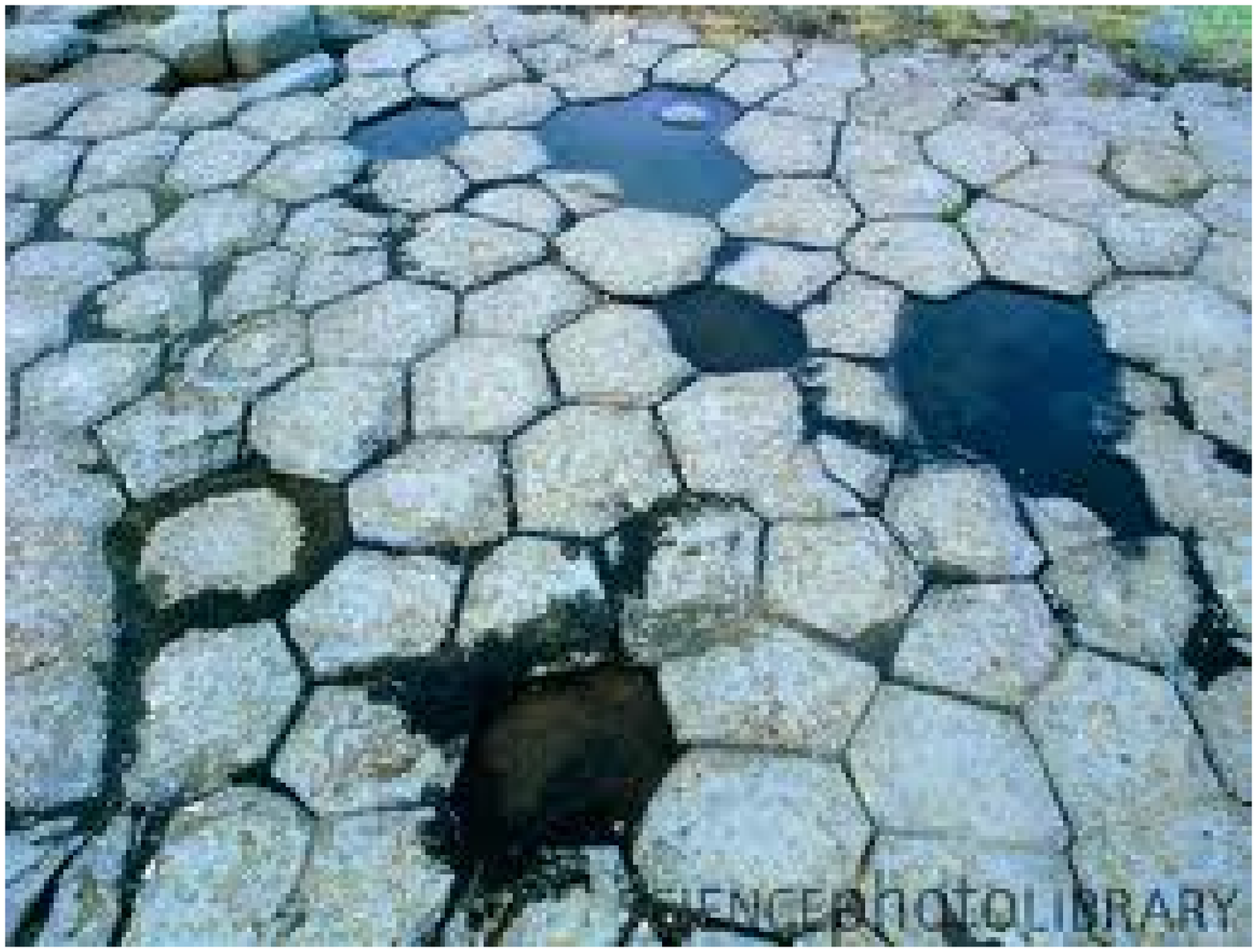}
   \includegraphics[width=\swid]{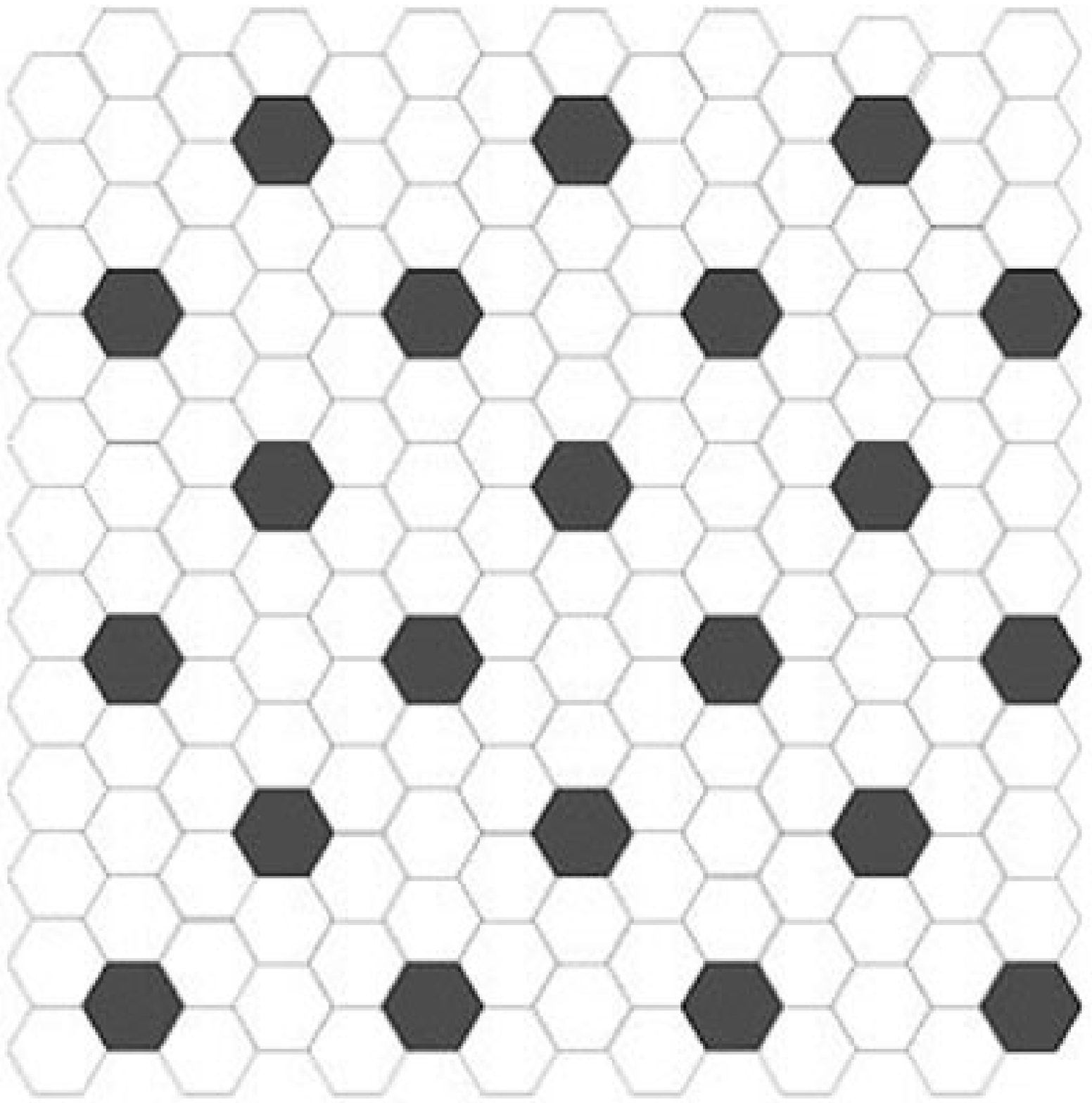}}
   \subfigure[half-scaled hex-tiling]{
\psset{unit=0.6cm} \pspicture(-3,-3)(3,3)
\pspolygon[fillstyle=solid,fillcolor=lightgray]
   (2,0)(1,-1.732)(-1,-1.732)(-2,0)(-1,1.732)(1,1.732)
\rput[l](0,-0.5){$H$}
\pspolygon[fillstyle=none]
   (1,0)(0.5,-0.866)(-0.5,-0.866)(-1,0)(-0.5,0.866)(0.5,0.866)
\pspolygon[fillstyle=none]
   (1,1.732) (0.5,2.598)(-0.5,2.598)(-1,1.732)(-0.5,0.866)(0.5,0.866)
\pspolygon[fillstyle=none]
   (-.5,0.866)(-1,0) (-2,0)(-2.5,0.866)(-2,1.732)(-1,1.732)
\pspolygon[fillstyle=none]
   (-.5,-0.866)(-1,0) (-2,0)(-2.5,-0.866)(-2,-1.732)(-1,-1.732)
\pspolygon[fillstyle=none]
   (1,-1.732) (0.5,-2.598)(-0.5,-2.598)(-1,-1.732)(-0.5,-0.866)(0.5,-0.866)
\pspolygon[fillstyle=none]
   (.5,-0.866)(1,0) (2,0)(2.5,-0.866)(2,-1.732)(1,-1.732)
\pspolygon[fillstyle=none]
   (.5,0.866)(1,0) (2,0)(2.5,0.866)(2,1.732)(1,1.732)
\endpspicture
}
\caption{ Examples for increased approximation error 
on a finer-scale tessellation.
(a) Image credit to Simon Fraser and Fastfloors.com.
(b) Non-nesting of the hex partition.
}
\label{fig:petal}
\end{figure}

Then the $L^2$ approximation error to $f$ from $\sH^0$ is
approximately $2\pi\eps$, the integral over the smoothing region. 
However, since $\sH^1$ does not 
contain a linear combination that can replicate $H$ (see \figref{fig:petal}b),
the $L^2$ approximation error to $f$ from $\sH^1$ is 
approximately $2\pi\hlf >> 2\pi\eps$.
That is the approximation error increases when refining the scale.
\end{example}

\section{Conclusion}
We identified simple necessary criteria for tessellations
to admit a refinable space of (convolutions of) \ifn s.
Lattices, in particular, must contain, for every facet, 
the whole hyperplane containing it.
With \corref{cor:voronoi}
we observed that the increased isotropy of the Voronoi cells of 
non-Cartesian root lattices prevents refinability.
Increased isotropy of the Voronoi cells is however
the main reason for considering non-Cartesian lattices in the first place:
they have higher packing densities leading to more efficient sampling
\cite{IC::PetersenM1962}.
We observed that even overcomplete spaces obtained by natural superposition 
of shifted hexagonal tessellations fail to provide refinable spaces
from convolution of \ifn s.
Finally, and here we omit the details, of the semi-regular tessellations
of the plane, only 3.6.3.6, the hex-tri-tessellation, satisfies 
the criteria of \propref{prop:straddle}; and while its \ifn s 
are refinable, generalizing the construction by convolution fails
to yield a family of higher-order splines sharing all good 
properties of box-splines.

In conclusion, if we want refinable classes of splines, remarkably few options 
exist apart of box splines and tensor-product B-splines.
Conversely, it should be noted, that adjusting and combining
the families of symmetric box-splines on crystallographic root lattices, 
exhibited and analyzed for example \cite{Kim:2011:SBS}, 
does yield splines for any level of smoothness 
that obey the underlying symmetries.

\medskip
{\bf Acknowledgement} 
Zhangjin Huang kindly worked out the first proof of non-refinability 
for \exref{ex:lozenge} and \exref{ex:propagate}, two scenarios
I posed. The current versions shorten his arguments. 

\bibliographystyle{alpha}
\bibliography{p}

\end{document}